\documentclass{amsart}[12pt]
\usepackage{amssymb,amsmath}
\usepackage{enumerate}
\usepackage[english]{babel}
\usepackage{color}

\newcommand {\ep} {\epsilon}

\newcommand {\gm} {\gamma}
\newcommand {\ii} {\infty}
\newcommand {\dt} {\delta}

\newcommand {\lb} {\lambda}
\newcommand {\Lb} {\Lambda}

\newcommand {\sm} {\setminus}
\newcommand {\su} {\subset}
\newcommand {\wt} {\widetilde}

\newcommand {\mc} {\mathcal}
\newcommand {\mb} {\mathbf}

\newtheorem{teo}{Theorem}[section]
\newtheorem{pro}{Proposition}[section]

\newtheorem{lm}{Lemma}[section]

\theoremstyle{definition}
\newtheorem{rem}{Remark}[section]
\newtheorem{df}{Definition}[section]

\title{Local ergodic theorems \\ in symmetric spaces of measurable operators}
\begin{document}
\date{May 7, 2018}

\begin{abstract}
Local mean and individual (with respect to almost uniform convergence in Egorov's sense) ergodic theorems are established
for actions of the semigroup $\mathbb R_+^d$ in symmetric spaces of measurable operators associated with a semifinite
von Neumann algebra.
\end{abstract}

\author{VLADIMIR CHILIN}
\address{National University of Uzbekistan, Tashkent, Uzbekistan}
\email{vladimirchil@gmail.com; chilin@ucd.uz}
\author{SEMYON LITVINOV}
\address{Pennsylvania State University \\ 76 University Drive \\ Hazleton, PA 18202, USA}
\email{snl2@psu.edu}

\maketitle

\section{Introduction}
Advancing Lance's extension of the pointwise ergodic theorem for actions of the group of integers on von
Neumann algebras, Conze and Dang-Hgoc \cite{cd} and Watanabe \cite{wa} studied continuous extensions
of Lance's results. In particular, noncommutative Wiener's local ergodic theorems were established for actions of
the semigroups $\mathbb R_+^d$ and $\mathbb R_+$, respectively.

Following Yeadon's ergodic theorem for the algebra $L^1$ of integrable operators associated with asemifinite von Neumann algebra, the corresponding Wiener's local ergodic theorem for actions of $\mathbb R_+$ with respect to bilaterally almost uniform convergence (in Egorov's sense) was initially considered in \cite{ca}. Later, Junge and Xu \cite{jx} derived that these averages converge bilaterally almost uniformly in any noncommutative $L^p$-space for $1\leq p<\ii$ and almost uniformly if
$2\leq p<\ii$. It was also noticed there that these results admit multiple versions.

We consider actions of the semigroup $ \mathbb R_+^d$ and show that, for a semifinite von Neumann algebra $\mc M$,
the corresponding ergodic averages $A_t(x)$ converge to $x$ almost uniformly as $t\to 0^+$ for all $x\in E$, where $E\su L^1(\mc M)+\mc M$ is a fully symmetric space such that the unit of $\mc M$ does not belong to $E$. Besides, we prove that if $E$ has order continuous norm $\|\cdot\|_E$, then $\|A_t(x)-x\|_E\to 0$ as $t\to 0^+$. Note that, along with any space $L^p(\mc M)$, $1\leq p<\ii$, the family of such fully symmetric spaces $E$ contains many noncommutative counterparts of classical Banach spaces of measurable functions, examples of which are given in the last section of the article.

\section{Preliminaries}

Let $\mathcal M$ be a semifinite von Neumann algebra equipped with a faithful normal semifinite trace $\tau$.
Let $\mathcal P(\mathcal M)$ be the complete lattice of projections in $\mathcal M$. If $\mathbf 1$ is the identity
of $\mathcal M$ and $e\in \mathcal P(\mathcal M)$, we write $e^{\perp}=\mathbf 1-e$.

Denote by $L^0=L^0(\mathcal M,\tau)$ the $*$-algebra of
$\tau$-measurable operators affiliated with $\mathcal M$. Let $\|\cdot\|_\ii$ be the uniform norm in $\mc M$. Endowed with the {\it measure topology} $t_\tau$ given by the
system of neighborhoods of zero
$$
\mc N(\ep,\dt)=\{ x\in L^0: \ \| xe\|_\ii\leq \dt \text{ \ for some \ } e\in \mc P(\mc M) \text{ \ with \ }
\tau(e^\perp)\leq \ep \},
$$
$\ep>0$, $\dt>0$, $L^0$ is a complete metrizable topological $*$-algebra \cite{ne}.

Let $L^p=L^p(\mathcal M,\tau)$, $1\leq p< \infty$,
$L^\infty(\mathcal M,\tau)=\mathcal M$, be the noncommutative $L^p$-space (see, for example, \cite{px})
equipped with the standard norm $\| \cdot \|_p$. Note that $\|x_n\|_p\to 0$ implies that $x_n\to 0$ in measure.

A net   $\{ x_\alpha\}_{\alpha \in A} \subset L^0$ is  said to converge to
$\widehat x\in L^0$ {\it almost uniformly} (a.u.) ({\it bilaterally almost uniformly} (b.a.u.))
if for every $\varepsilon>0$ there exists
$e\in \mathcal P(\mathcal M)$ such that $\tau(e^{\perp})\leq \varepsilon$ and
$\lim\limits_{\alpha \in A}\| (\widehat x-x_\alpha)e\|_\ii = 0$ (respectively,
$\lim\limits_{\alpha \in A}\| e(\widehat x-x_\alpha)e\|_\ii= 0$). Note that a.u. convergence is generally
stronger than b.a.u. convergence. Also, it is not difficult to verify that $x_n\to 0$ in measure implies that $x_{n_k}\to 0$
a.u. for some subsequence $\{x_{n_k}\}\su \{x_n\}$.

A linear map $T: L^1+\mathcal M \rightarrow  L^1+\mathcal M$ is called a
{\it Dunford-Schwartz operator} if
$$
\|T(x)\|_1 \leq \|x \|_1 \text{ \ for all \ } x \in L^1 \text{ \ and \ }
\|T(x)\|_\ii \leq \|x \|_\ii \text{ \ for all \ } x \in \mathcal M.
$$
Denote by  $DS^+=DS^+(\mathcal M, \tau)$ the set of positive Dunford-Schwartz operators.

Let $\mathbb N$ (respectively, $\mathbb R$) be the set of natural (respectively, real) numbers. Fix $d\in \mathbb N$ and denote
$\mathbb R_+^d=\{ \mb u=(u_1,\dots,u_d):\ u_i\ge 0, \ i=1,\dots,d\}$.
In what follows, $\{T_{\mb u}\}_{\mb u\in \mathbb R^d_+}\subset DS^+$ is a semigroup
such that $T_{\mb 0}(x)=x$ for all $x\in L^1+\mc M$.
If $\mb u, \mb v\in \mathbb R^d_+$, $\mb u=(u_1,\dots,u_d)$, $\mb v=(v_1,\dots,v_d)$,
then $\mb u\to\mb v$ means that $u_i\to v_i$ for each $i=1,\dots,d$.

A semigroup $\{T_{\mb u}\}_{\mb u\in \mathbb R^d_+}$ is said to be {\it strongly continuous} on
$ L^p$, $1\leq p<\ii$, if
$$
 \lim\limits_{\mb u\to \mb v}\|T_{\mb u}(x)-T_{\mb v}(x)\|_p=0.
$$
for each $x\in L^p$.
\begin{pro}\label{p21}
If a semigroup $\{T_{\mb u}\}_{\mb u\in \mathbb R^d_+}\su DS^+$ is strongly continuous on $L^1$, then it is
also strongly continuous on $L^p$ for every $1<p<\ii$.
\end{pro}
\begin{proof}
It is sufficient to show that $\| T_{\mb u_n}(x)-T_{\mb u_0}(x)\|_p\to 0$ as $\mb u_n\to \mb u_0$ for   any $0\neq x\in L^p_+$.
Fix $\ep>0$. If $x=\int_0^\ii\lb de_\lb$ is the spectral decomposition of $x$ and
$m\in \mathbb N$, denote $x_m=\int_0^{1/m}\lb de_\lb+\int_m^\ii\lb de_\lb$, $y_m=x-x_m$ and choose
$m_0$ to satisfy $\|x_{m_0}\|_p\leq \frac\ep4$. As $T_{\mb u}$ is a contraction in $L^p$
for each $\mb u\in \mathbb R_+^d$, we have
$$
\| T_{\mb u_n}(x_{m_0})-T_{\mb u_0}(x_{m_0})\|_p\leq\frac\ep2 \text{ \ \ for all\ \ } n.
$$
Next, we have $y_{m_0}\in L^1_+\cap \mc M$.  Since $x\neq 0$, without
loss of generality, we may assume that $y_{m_0}\neq 0$.  Let
$z_n=\|y_{m_0}\|_\ii^{-1}(T_{\mb u_n}(y_{m_0})-T_{\mb u_0}(y_{m_0}))$.
By the assumption, $\tau(|z_n|)=\| z_n\|_1\to 0$. Besides, since $T_{\mb u}$ is a contraction in
$\mc M$ for each $\mb u\in \mathbb R_+^d$,
it follows that $-\mb 1\leq z_n\leq \mb 1$, and so $|z_n|^p\leq |z_n|$ for all $n$.
Therefore, we have
$$
\|z_n\|_p=\tau(|z_n|^p)^{1/p}\leq \tau(|z_n|)^{1/p}\to 0,
$$
hence
$$
\lim_{\mb u_n\to\mb u_0}\| T_{\mb u_n}(y_{m_0})-T_{\mb u_0}(y_{m_0})\|_p= 0.
$$
Then there is $N\in \mathbb N$ such that
$$
\| T_{\mb u_n}(y_{m_0})-T_{\mb u_0}(y_{m_0})\|_p\leq\frac\ep2 \text{ \ \ for all\ \ } n\ge N.
$$
Therefore
$$
\| T_{\mb u_n}(x)-T_{\mb u_0}(x)\|_p\leq \| T_{\mb u_n}(x_{m_0})-T_{\mb u_0}(x_{m_0})\|_p+
\| T_{\mb u_n}(y_{m_0})-T_{\mb u_0}(y_{m_0})\|_p\leq\ep
$$
whenever $n\ge N$, which completes the proof.
\end{proof}

In view of Proposition \ref{p21}, we assume, throughout, that
$\{T_{\mb u}\}_{\mb u\in \mathbb R^d_+}\su DS^+$ is a strongly continuous semigroup on $L^1$.
Then, for a given $1\leq p<\ii$ and $x\in L^p$, the averages
\begin{equation}\label{eq31}
A_t(x)=\frac 1{t^d}\int_{[0,t]^d}T_{\mb u}(x)d\mb u
\end{equation}
belong to  $L^p$ for every $t>0$.


\section{Local Mean Ergodic Theorem in $L^p(\mc M,\tau)$, $1\leq p<\ii$}

\begin{lm}\label{l33}
Let  $1\leq p<\ii$, and let $t_0>0$. If $y\in L^p$ and $x=A_{t_0}(y)$, then $\|A_t(x)-x\|_p\to 0$ as $t\to 0$.
 In addition, if $y\in L^p\cap \mc M$, then $\|A_t(x)-x\|_\ii\to 0$ as $t\to 0$.
\end{lm}
\begin{proof}
 Let $0<v_1,\dots,v_d\leq v<t_0$ and $\mb v=(v_1, \dots, v_d)\in \mathbb R^d_+$. Then
\begin{equation*}
\begin{split}
T_{\mb v}(x)&-x=\frac 1{t_0^d}\bigg [\int\limits_{[0,t_0]^d}T_{\mb u+\mb v}(y)d\mb u
-\int\limits_{[0,t_0]^d}T_{\mb u}(y)d\mb u \bigg ]
\\
&=\frac 1{t_0^d} \bigg[\int\limits_{[v_1,t_0+v_1]\times \dots \times [v_d,t_0+v_d] }T_{\mb u}(y)d\mb u
-\int\limits_{[0,t_0]^d}T_{\mb u}(y)d\mb u\bigg ]
\\
&=\frac 1{t_0^d}\bigg [\int\limits_{[v_1,t_0]\times \cdots \times [v_d,t_0]}T_{\mb u}(y)d\mb u
+\int\limits_{[v_1,t_0+v_1]\times \dots \times [v_d,t_0+v_d]\sm ([v_1,t_0]\times \dots \times [v_d,t_0])}
T_{\mb u}(y)d\mb u
\\
&\ \ \ \ \ \ \ \ -\int\limits_{[v_1,t_0]\times \dots \times [v_d,t_0] }T_{\mb u}(y)d\mb u
-\int\limits_{[0,t_0]^d\sm ([v_1,t_0]\times \dots \times [v_d,t_0])}T_{\mb u}(y)d\mb u\bigg ]
\\
&=\frac 1{t_0^d}\bigg [\int\limits_{[v_1,t_0+v_1]\times \dots \times [v_d,t_0+v_d]\sm ([v_1,t_0]\times
\cdots \times [v_d,t_0])}T_{\mb u}(y)d\mb u\\
&\ \ \ \ \ \ \ \ -\int\limits_{[0,t_0]^d\sm ([v_1,t_0]\times \dots \times [v_d,t_0])}T_{\mb u}(y)d\mb u\bigg ],
\end{split}
\end{equation*}
implying that
\[
\|T_{\mb v}(x)-x\|_p\leq 2\frac{t_0^d-(t_0-v)^d}{t_0^d}\|y\|_p.
\]
As $0<v<t_0$, we have
\[
t_0^d-(t_0-v)^d=v\sum_{k=1}^d\binom{d}{k}t_0^{d-k}(-v)^{k-1}<v\sum_{k=1}^d\binom{d}{k}t_0^{d-k}v^{k-1}
<v\sum_{k=1}^d\binom{d}{k}t_0^{d-1},
\]
which implies that, with $C(t_0,d)=2t_0^{-1}\sum_{k=1}^d\binom{d}{k}$,
\[
\|T_{\mb v}(x)-x\|_p<v\,C(t_0,d)\|y\|_p\leq t\,C(t_0,d)\|y\|_p
\]
whenever $v\leq t<t_0$. It follows then that
\[
\| A_t(x)-x\|_p=\bigg \|\frac 1{t^d}\int_{[0,t]^d}(T_{\mb v}(x)-x)d\mb v\bigg \|_p\leq  t\,C(t_0,d)\|y\|_p,
\]
hence $\| A_t(x)-x\|_p\to 0$ as $t\to 0$.

Second part of the statement follows by replacing the norm $\|\cdot\|_p$ with $\|\cdot\|_\ii$ in the
above inequalities.
\end{proof}

Now we can prove a local mean ergodic theorem.  One may notice that, unlike ergodic theorems with
"infinite time", the convergence holds also in the space $L^1$ with $\tau(\mathbf 1) = \infty$.

\begin{teo}\label{t33b}
Let  $1\leq p<\ii$. If $x\in L^p$, then $\| A_t(x)-x\|_p\to 0$ as $t\to 0$.
\end{teo}
\begin{proof}

Since $\| T_{\mb u}(x)-x\|_p\to 0$ as $\mb u\to \mb 0$ (see Proposition \ref{p21}), given $n\in \mathbb N$, there is $t_n>0$ such that
\[
\| T_{\mb u}(x)-x\|_p<\frac 1n
\]
 whenever $\mb u = (u_1,\dots,u_d), \ 0 < u_1,\dots,u_d \leq t_n$. Then we have
\[
\| A_{t_n}(x)-x\|_p=\bigg \| \frac 1{t_n^d}\int\limits_{[0,t_n]^d}(T_{\mb u}(x)-x)d\mb u \bigg \|_p<\frac 1n,
\]
hence $x_n=A_{t_n}(x)\to x$ in $L^p$.

Fix $\ep>0$. As $\|x_n-x\|_p\to 0$, choose
$n_0$ such that $\|x_{n_0}-x\|_p<\frac\ep4$. Next, since, by Lemma \ref{l33}, $\|A_t(x_{n_0})-x_{n_0}\|_p\to 0$,
there is $t_0>0$ such that $0<t<t_0$ entails that $\|A_t(x_{n_0})-x_{n_0}\|_p<\frac\ep2$. Now, given $0<t<t_0$,
we have
\begin{equation*}
\begin{split}
\|A_t(x)-x\|_p&=
\|A_t(x)-A_t(x_{n_0})\|_p+\|A_t(x_{n_0})-x_{n_0}\|_p+\|x_{n_0}-x\|_p\\
&< 2\|x_{n_0}-x\|_p+\frac\ep2<\ep,
\end{split}
\end{equation*}
and the assertion follows.
\end{proof}

\section{Maximal Ergodic Inequality for Actions of $\mathbb R_+^d$}

The next maximal ergodic inequality, due toYeadon \cite[Theorem 1]{ye}, provides the main tool for proving individual ergodic theorems in semifinite von Neumann algebras.
\vskip 5 pt
\begin{teo}\label{t31}
Let  $T\in DS^+$. Then for every
$x \in L^1_+$ and $\lb>0$ there exists $e\in \mc P(\mc M)$ such that
\[
\tau(e^\perp)\leq \frac {\| x\|_1} \lb
\text{ \ \ and \ \ } \sup_n \bigg \| e\,\frac 1n\sum_{k=0}^{n-1}T^k(x)\,e\bigg \|_\ii \leq \lb.
\]
\end{teo}

Here is a continuous multi-parameter extension of Theorem \ref{t31} (cf. \cite[Remark 4.7]{jx}):
\vskip 5 pt
\begin{teo}\label{t32}
 Let the averages $A_t$ be given by (\ref{eq31}). Then there is a constant $\chi_d>0$ such that, given $x\in L^1_+$ and $\lb>0$, there exists $e\in \mc P(\mc M)$ satisfying inequalities
\[
\tau(e^\perp)\leq \frac {2\| x\|_1} \lb
\text{ \ \ and \ \ } \sup_{t> 0} \| eA_t(x)e \|_\ii \leq \chi_d \, \lb.
\]

\end{teo}
\begin{proof}
Let $\frac nm\in \mathbb Q_+$, where $n, m \in \mathbb N$. Let $x\in L^1_+$ and denote
$y_m=\int_{[0,1]^d}T_{\frac{\mb v}m}(x)d\mb v$. Then we have
\begin{equation}\label{e5}
\begin{split}
A_{\frac nm}(x)&=\frac {m^d}{n^d}\int_{[0,n/m]^d}T_{\mb u}(x)d\mb u
=\frac {1}{n^d } \int_{[0,n]^d}T_{\frac{\mb v}m}(x)d\mb v \\
&=\frac {1}{n^d } \sum_{i_1=0}^{n-1}\dots\sum_{i_d=0}^{n-1} \int\limits_{i_1}^{i_1+1}  T_{(v_1/m,0,\dots,0)}(x)dv_1 \ \dots \int\limits_{i_d}^{i_d+1}  T_{(v_d/m,0,\dots,0)}(x)dv_d\\
&=\frac {1}{n^d }  \sum_{i_1=0}^{n-1}\dotsc\sum_{i_d=0}^{n-1}T_{(1/m,0,\dots,0)}^{i_1}\cdot
\dotsc \cdot T_{(0,\dots,0,1/m)}^{i_d}(y_m).
\end{split}
\end{equation}

Next we use a result due to A. Brunel \cite{br} (see \cite[Ch.6, \S 6.3, Theorem 3.4]{kr}) (although it was originally formulated for commuting contractions in a commutative $L^1$-space, one can see that the proof goes when $L^1=L^1(\mc M,\tau)$): if $T_1,\dots,T_d$ are positive commuting contractions of $L^1$, then there exist $\chi_d>0$, $n_d\in\mathbb N$, and
$\{ a_{\mb n}>0: \mb n\in \mathbb N_0^d\}$ with $\sum a_{\mb n}=1$ such that the operator
$S=\sum a_{\mb n}T_1^{n_1}\cdot \dotsc \cdot T_d^{n_d}$ satisfies the inequality
\begin{equation}\label{e6}
\frac 1{n^d}\sum_{i_1=0}^{n-1}\dotsc \sum_{i_d=0}^{n-1}T_1^{i_1}\cdot \dotsc \cdot T_d^{i_d}(x)\leq
\frac {\chi_d}{n_d}\sum_{j=0}^{n_d-1}S^j(x)
\end{equation}
for all $n=1,2,\dots$ and $x\in L_+^1$.

Now it follows from (\ref{e5}) and (\ref{e6}) that
\begin{equation}\label{e7}
0\leq A_{\frac nm}(x)\leq {\chi_d}\frac 1{n_d}\sum_{k=0}^{n_d-1}S^k(y_m).
\end{equation}
Since $S\in DS^+$ it follows by Theorem \ref{t31} that there is $f\in \mc P(\mc M)$ such that
$$
\tau(f^\perp)\leq \frac {\| y_m\|_1}\lb\leq \frac {\| x\|_1}\lb \text{ \ \ and \ \ } \sup_n \bigg \| f\, \frac 1n\sum_{k=0}^{n-1} S^k(y_m)\, f \bigg \|_\ii\leq \lb,
$$
implying, in view of (\ref{e7}), that
$$
\sup_{r\in \mathbb Q_+\sm \{0\}} \| f\, A_r(x)\, f\|_\ii \leq \chi_d\lb.
$$

If $t>0$, let $t \leftarrow r_n \in \mathbb Q_+\sm \{0\}$. Then we have $A_{r_n}(x)\to A_t(x)$  in measure. Therefore,
 $A_{r_{n_k}}(x)\to A_t(x)$ a.u. for a subsequence $\{r_{n_k}\}\su \{r_n\}$. Thus, it is possible to find
$g\in \mc P(\mc M)$ such that
$$
\tau(g^\perp)\leq \frac {\| x\|_1}\lb \text { \ \ and \ \ } \|gA_{r_{n_k}}(x)g\|_\ii \to \| gA_t(x)g\|_\ii \text{ \ as \ } k\to \ii.
$$
Letting $e=f\land g$, we obtain the required inequalities.
\end{proof}

\begin{rem}\label{r31}
As it was carried out in \cite{cl1}, one can see that Theorem \ref{t32} can be extended to an arbitrary space $L^p$,
$1<p<\ii$.
\end{rem}

\section{Local Individual Ergodic Theorem  in $L^p(\mc M,\tau)$}

We utilize the notion of bilaterally uniform equicontinuity in measure which is a noncommutative counterpart of the continuity in measure at zero of the maximal function of a sequence of maps from a normed space into a space of almost everywhere bounded measurable functions (see \cite{li}):

\begin{df}
Let $(X,\| \cdot \|)$ be a normed space. A sequence of maps $M_n: X\to L^0$
is called {\it bilaterally uniformly equicontinuous in measure (b.u.e.m.) at zero}
if for every $\ep>0$ and $\dt>0$ there exists $\gm>0$ such that, given $x\in X$ with $\| x\|<\gm$,
there is a projection $e\in \mc P(\mc M)$ satisfying conditions
$$
\tau(e^{\perp})\leq \ep \text{ \ \ and \ \ } \sup_n\| eM_n(x)e\|_{\ii}\leq \dt.
$$
\end{df}

In order to establish a.u. convergence of the averages (\ref{eq31}), we will need the following.
\begin{lm}\label{l31}
Let $(X,\| \cdot \|)$ be a normed space, and let $M_n:X\to L^0$ be a sequence of maps b.u.e.m. at zero.
If $\{z_m\}\su X$ is such that $\| z_m\|\to 0$, then for every $\ep>0$ and $\dt>0$ there are $z_{m_0}\in \{z_m\}$
and $e\in \mc P(\mc M)$ satisfying conditions
$$
\tau(e^\perp)\leq \ep \text{ \ \ and \ \ } \sup_n\| M_n(z_{m_0})e\|_\ii \leq \dt.
$$
\end{lm}
\begin{proof}
Since $\| z_m\| \to 0$  and $\{M_n\}$  is b.u.e.m. at zero in $(X, \|\cdot\|)$, for every $n,k\in \mathbb N$
there exist $z_{n,k} \in \{z_{m}\}$ and $g_{n,k} \in \mc P(\mc M)$ such that
$$
\tau(g_{n,k}^\perp)\leq \ \frac \ep{2^{n+k+1}}  \text{ \ \ and \ \ } \sup_m\| g_{n,k} M_m(z_{n,k})g_{n,k}\|_\ii\leq \dt.
$$
 In particular,
$$
\| g_{n,k} M_n(z_{n,k})g_{n,k}\|_\ii\leq \dt \text{ \ \ for all\ \ }n,k.
$$

If $\mb l(y)$ ($\mb r(y)$) is the left (respectively, right) support of an operator $y\in L^0$
and $q_{n,k}=\mb 1-\mb r(g_{n,k}^\perp M_n(z_{n,k}))$, then
$$
\tau(q_{n,k}^\perp)=\tau\left (\mathbf r(g_{n,k}^\perp M_n(z_{n,k}))\right )=
\tau\left (\mathbf l (g_{n,k}^\perp M_n(z_{n,k}))\right )\leq \tau(g_{n,k}^\perp)\leq \frac \ep{2^{n+k+1}}.
$$
Besides,
$$
 M_n(z_{n,k}) q_{n,k} =g_{n,k} M_n(z_{n,k}) q_{n,k} +g_{n,k}^\perp  M_n(z_{n,k}) q_{n,k}
= g_{n,k}  M_n(z_{n,k}) q_{n,k}.
$$
Therefore, letting $e_{n,k} = g_{n,k} \land q_{n,k}$, we obtain $\tau(e_{n,k}^\perp)\leq  \frac \ep{2^{n+k}}$ and
$$
M_n(z_{n,k}) e_{n,k}= M_n(z_{n,k})q_{n,k} e_{n,k} = g_{n,k}M_n(z_{n,k}) q_{n,k}e_{n,k}
= g_{n,k}M_n(z_{n,k}) g_{n,k} e_{n,k},
$$
implying that
$$
\| M_n(z_{n,k}) e_{n,k}\|_\ii\leq \|g_{n,k}M_n(z_{n,k}) g_{n,k} \|_\ii\leq \dt
$$
for all $n$ and $k$.

If $e=\bigwedge\limits_{n,k} e_{n,k}$, then we have
$$
\tau(e^\perp)\leq \ep \text{ \ \ and \ \ } \sup_n\| M_n(z_{n,k}) e\|_\ii\leq \dt \text{ \ \ for all \ } k,
$$
and the result follows.
\end{proof}

We will also need the next two lemmas.
\begin{lm}\label{l333}
Let $1\leq p\leq\ii$, and let $t>0$. If $x \in L^p$, then
\[
\lim\limits_{s\to t}\| A_t(x)-A_s(x)\|_p\to 0.
\]
\end{lm}
\begin{proof}
If $0<s< t$, then
\begin{equation*}
\begin{split}
\| A_t(x)-A_s(x)\|_p &\leq\bigg\|\bigg(\frac 1{t^d}-\frac 1{s^d}\bigg)\int\limits_{[0,s]^d}T_{\mb u}(x)d\mb u \bigg\|_p+ \bigg\|\frac 1{t^d}\int\limits_{[0,t]^d \setminus[0,s]^d}T_{\mb u}(x)d\mb u\bigg\|_p \\
&\leq \bigg(\frac 1{s^d}-\frac 1{t^d}\bigg)\int\limits_{[0,s]^d}\|T_{\mb u}(x)\|_p d\mb u + \frac 1{t^d}\int\limits_{[0,t]^d \setminus[0,s]^d}\|T_{\mb u}(x)\|_p d\mb u \\
&\leq 2\frac{t^d-s^d}{t^d}\|x\|_p,
\end{split}
\end{equation*}
implying that $\lim\limits_{s\to t^-}\| A_t(x)-A_s(x)\|_p\to 0$. Convergence $\lim\limits_{s\to t^+}\| A_t(x)-A_s(x)\|_p\to 0$ is proved similarly.
\end{proof}

\begin{lm}\label{l32}
Let $t_n\to 0^+$.
\begin{enumerate}[(i)]
\item If $1<p<\ii$, then $A_{t_n}(x)\to x$ a.u. for all $x\in L^p$;
\item $A_{t_n}(x)\to x$ b.a.u. for all $x\in L^1$.
\end{enumerate}
\begin{proof}
(i) By Theorem \ref{t33b}, we have $\|A_{t_n}(x)-x\|_p\to 0$ for each $x\in L^p$. This implies that the set
$$
\mc D=\{A_{t_n}(x):\ n\in \mathbb N, \ x\in L^p\cap \mc M\}
$$
is dense in $L^p\cap \mc M$, hence in $L^p$.

By Lemma \ref{l33}, $\|A_{t_n}(y)-y\|_\ii\to 0$  whenever $y\in\mc D$.
Consequently, $A_{t_n}(y)\to y$ a.u. for each $y\in \mc D$.
Since, in view of Remark \ref{r31}, the sequence $\{A_{t_n}\}$ is b.u.e.m. at zero on $L^p$, it follows by \cite[Proposition 3.1]{cl3}
(with $X=L^p$ and $M_n=A_{t_n}$)  that the set $\{ x \in L^p:\ \{A_{t_n}(x)\} \text{\ converges a.u.}\}$ is closed in $L^p$. As $\mc D$ is dense in $L^p$, we conclude that the sequence $\{A_{t_n}(x)\}$ converges a.u.
for each $x\in L^p$. Now, given $x\in L^p$, $A_{t_n}(x)\to x$ in $L^p$ entails that  $A_{t_n}(x)\to x$ in measure.
Therefore $A_{t_{n_k}}(x)\to x$ a.u. for some subsequence $\{t_{n_k}\}$, implying that $A_{t_n}(x)\to x$ a.u.

(ii) Let $x\in L^1$, $1<p<\ii$, and let $\{x_m\}\su L^p$ be such that $\|x-x_m\|_1\to 0$. Since $x_m\to x$
in measure, we can assume without loss of generality that $x_m\to x$ a.u. also. Fix $\ep>0$ and $\dt>0$. There is
$f\in \mc P(\mc M)$ with $\tau(f^\perp)\leq \frac \ep3$ and $N_1\in \mathbb N$ such that
$$
\|(x-x_m)f\|_\ii<\frac\dt3 \text{ \ \ for all\ \ } m\ge N_1.
$$
Next, by Theorem \ref{t32}, the sequence $\{A_{t_n}\}$ is b.u.e.m. at zero on $L^1$. This implies that there exists
$N_2\in \mathbb N$ such that $m\ge N_2$ implies that there is $h\in \mc P(\mc M)$ such that $\tau(h^\perp)\leq\frac\ep3$
and
$$
\sup_n\|h(A_{t_n}(x-x_m)h\|_\ii<\frac\dt3  \text{ \ \ for all\ \ } m\ge N_2.
$$
Let $m_0\ge \max\{N_1,N_2\}$. Since, in view of (i), $A_{t_n}(x_{m_0})\to x_{m_0}$ a.u., it follows that there is
$g\in \mc P(\mc M)$ with $\tau(g^\perp)\leq \frac\ep3$ and $N\in \mathbb N$ satisfying
$$
\|(A_{t_n}(x_{m_0})-x_{m_0})g\|_\ii< \frac\dt3 \text{ \ \ for all\ \ } n\ge N.
$$
Now, letting $e=f\land g\land h$, we obtain $\tau(e^\perp)\leq\ep$ and
\begin{equation*}
\begin{split}
\|e(A_{t_n}(x)-x)e\|_\ii&\leq \|eA_{t_n}(x-x_{m_0})e\|_\ii+\|e(A_{t_n}(x_{m_0})-x_{m_0})e\|_\ii\\
&+\|e(x-x_{m_0})e\|_\ii<\dt
 \end{split}
\end{equation*}
whenever $n\ge N$, which implies that $A_{t_n}(x)\to x$ b.a.u.
\end{proof}
\end{lm}

Now we give an improvement of \cite[Theorem 6.8 i) b)]{jx}; see also Remarks after \cite[Theorem 6.8]{jx}.

\begin{teo}\label{t33} For every $x\in L^1$ the averages (\ref{eq31})
converge a.u. to $x$ \ as \ $t \to 0$.
\end{teo}
\begin{proof}
Show first that $A_t(x)\to x$ a.u. if $x\in L^1\cap \mc M$. Fix $1<p<\ii$.
Since $L^1\cap \mc M\su L^p$, Lemma \ref{l32} implies that $A_{1/n}(x)\to x$ a.u.
In particular,
\begin{equation}\label{e9}
A_{[1/t]^{-1}}(x)\to x \text{ \ a.u. \ as \ }t\to 0.
\end{equation}
Next, we have
\begin{equation}\label{e8}
\begin{split}
\| A_t(x)&-A_{[1/t]^{-1}}(x)\|_\ii
=\left \| \left [\frac 1t\right ]^d\int_{[0, [1/t]^{-d}]}T_{\mb u}(x)d\mb u-
\frac 1{t^d}\int_{[0,t^d]}T_{\mb u}(x)d\mb u \right \|_\ii \\
&=\left \| \left [\frac 1t\right ]^d\int_{[0, [1/t]^{-d}]\sm [0,t^d]}T_{\mb u}(x)d\mb u+
\left (\left [\frac 1t\right ]^d-\frac 1{t^d}\right )\int_{[0,t^d]}T_{\mb u}(x)d\mb u \right \|_\ii \\
&\leq \left [\frac 1t \right ]^d \left ( \left [\frac 1t \right ]^{-d}-t^d\right )\| x\|_\ii+
\left (\left [\frac 1t\right ]^d -\frac 1{t^d} \right )t^d\| x\|_\ii\to 0
\end{split}
\end{equation}
as $t\to 0$. Now, by (\ref{e9}), given $\ep>0$, there exists $e\in \mc P(\mc M)$ such that
$\tau(e^\perp)\leq \ep$ and
$$
 \|(A_{[1/t]^{-1}}(x)-x)e\|_\ii\to 0 \text{ \ \ as \ }t\to 0.
$$
Therefore, taking into account (\ref{e8}), we obtain
$$
\|(A_t(x)-x)e\|_\ii\leq \| (A_t(x)-A_{[1/t]^{-1}}(x))e\|_\ii+\|(A_{[1/t]^{-1}}(x)-x)e\|_\ii\to 0
$$
as $t\to 0$, hence $A_t(x)\to x$ a.u.

Let $x\in L^1_+$, and let $\{e_\lb\}_{\lb\ge 0}$ be its spectral family. If $x_m=\int_0^m\lb de_\lb$ and
$z_m=x-x_m$ for a positive integer  $m$, then $\{x_m\}\su L^1_+\cap \mc M$, $\{ z_m\}\su L^1_+$
and $\|z_m\|_1\to 0$.

Fix $\ep>0$ and $\dt>0$.
If $\{r_n\}$ is a sequence of all positive rational numbers, then, by Theorem \ref{t32}, the sequence $\{A_{r_n}\}$
is b.u.e.m. at zero on $L^1_+$, hence on $L^1$. Then, applying Lemma \ref{l31},
we find a projection $e\in \mc P(\mc M)$ and $z_{m_0}\in \{z_m\}$ such that
$$
\tau(e^\perp)\leq \frac \ep2 \text{ \ \ and \ \ } \sup_n\|A_{r_n}(z_{m_0})e\|_\ii <\frac \dt3.
$$

If $t>0$, then $r_{n_k}\to t$ for some subsequence $\{r_{n_k}\}$, so $\|A_t(z_{m_0})-A_{r_{n_k}}(z_{m_0})\|_1\to 0$ by Lemma \ref{l333}. Therefore $A_{r_{n_k}}(z_{m_0})\to A_t(z_{m_0})$ in measure, which implies that there is a subsequence $\{r_{n_{k_l}}\}$ such that $A_{n_{k_l}}(z_{m_0})\to A_t(z_{m_0})$ a.u.

Since $\|A_{n_{k_l}}(z_{m_0})e\|_\ii<\frac \dt 3$ for each $l$, it follows from \cite[Lemma 5.1]{cl3} that
\begin{equation}\label{eq2}
\sup_{t>0}\|A_t(z_{m_0})e\|_\ii \leq \frac \dt 3.
\end{equation}
Because $x_{m_0} \in  L^1 \cap \mc M$, we have $A_{t}(x_{m_0} )\to x_{m_0}$ a.u., so
the net $A_{t}(x_{m_0})$ is a.u. Cauchy.
Therefore, there exist $g \in \mc P(\mc M)$ and $t_0 > 0$ such that
\begin{equation}\label{eq3}
\tau(g^\perp)< \frac {\varepsilon}{2} \text { \ \ and \ \ }
\|(A_t( x_{m_0})- A_{t'}( x_{m_0}))g\|_\ii < \frac {\delta}{3}.
\end{equation}
for all $0 < t, t' < t_0$.

If $h = e \wedge g$, then $\tau(h^\perp)< \ep $ and, in view of (\ref{eq2})  and (\ref{eq3}), we have
\begin{equation*}
\begin{split}
\| (A_t(x)- A_{t'}(x))h\|_\ii &\leq \| (A_t(x_{m_0})- A_{t'}( x_{m_0}))h\|_\ii \\
&+\| A_t(z_{m_0})h\|_\ii + \| A_{t'}( z_{m_0})h\|_\ii < \delta.
\end{split}
\end{equation*}
Thus, the net $A_{t}(x)$ is a.u. Cauchy as $t \to 0$. By the proof of Theorem 2.3
in \cite{cls}, $L^0$ is complete with respect to a.u. convergence.  Therefore,  the net  $\{A_t(x)\}$ converges a.u. in $L^0$. Taking into account that, by Lemma \ref{l32}, $A_{1/n}(x)\to x$ b.a.u.,
we conclude that $A_t(x)\to x$ a.u. as $t\to 0$ for all $x\in L^1_+$, hence for all $x\in L^1$.
\end{proof}

\begin{rem}
In view of Lemma \ref{l32}, it is clear that the assertion of Theorem \ref{t33} holds for all $1\leq p<\ii$.
Alternatively, see the proof of Theorem \ref{t34} below.
\end{rem}

\section{Extension to Fully Symmetric Spaces of Measurable Operators}
The  {\it non-increasing rearrangement} of $x \in L^0$ is defined as
$$
\mu_t(x)=\inf\{\lambda>0: \ \tau\{|x|>\lambda\}\leq t\}, \ \ t> 0,
$$
where $|x|=(x^*x)^{1/2}$, the absolute value of $x$ (see, for example, \cite{fk}).

Let $L^0_\tau$ be the $*$-subalgebra of $L^0$ consisting of such $x \in L^0$ that
$\tau\{|x|>\lambda\}<\ii$ for some $\lambda >0$. The collection of sets
$$
\mc N(\ep,\dt) =\left\{x \in L^0(\tau):\ \mu_\dt (x) \le \ep\right\}, \ \ \ep>0, \ \dt>0.
$$
forms a basis of neighborhoods of zero for the measure topology $t_\tau$ in $L^0_\tau$. Note that $\mc M$
is dense in $(L^0_\tau, t_\tau)$ \cite{ne}.

A non-zero linear  subspace  $E\su L^0_\tau$ with a Banach norm  $\| \cdot \|_E$ is called {\it fully symmetric} if conditions
\[
x\in E, \ y\in L^0_\tau, \ \int \limits_0^s\mu_t(y)dt\leq  \int \limits_0^s\mu_t(x)dt \text{ \ for all\ }s>0 \ \ (\text{writing} \ \ x \prec\prec y)
\]
imply that $y\in E$ and $\| y\|_E\leq \| x\|_E$.

Let $L^0(0,\ii)$ be the linear space of (equivalence classes of) almost everywhere finite complex-valued
Lebesgue measurable functions on the interval $(0,\ii)$. We identify $L^{\ii}(0,\ii)$ with the commutative
von Neumann algebra acting on the Hilbert space $L^2(0,\ii)$ via
multiplication by the elements from $L^{\ii}(0,\ii)$ with the trace given by the integration
with respect to the Lebesgue measure $\mu$. A   fully symmetric space $E\su L^0(L^\ii(0,\ii), \nu)$, where
the trace $\nu$ is given by the Lebesgue integral with respect to measure $\mu$, is called a
{\it fully symmetric function space} on $(0,\ii)$ (see, for example, \cite{kps}).

If $E=E(0,\ii)$ is a fully symmetric function space, define
\[
E(\mc M)=E(\mc M, \tau)=\{ x\in L^0: \ \mu_t(x)\in E\}
\]
and set
\[
\| x\|_{E(\mc M)}=\| \mu_t(x)\|_E,  \ x\in E(\mc M).
\]
It is shown in \cite{ks} that $(E(\mc M), \| \cdot \|_{E(\mc M)})$ is a fully symmetric space.
If $1\leq p<\ii$ and $E=L^p(0,\ii)$, the space $(E(\mc M), \| \cdot \|_{E(\mc M)})$ coincides with the noncommutative
$L^p$-space $L^p=L^p(\mc M)=(L^p(\mc M, \tau), \| \cdot \|_p)$ \ \cite{ye0}.
 In addition, $L^{\infty}(\mc M)=\mc M$ and
\[
(L^1\cap L^\ii)(\mc M)=L^1(\mc M)\cap\mc M \text{ \ \ with \ \ } \|x\|_{L^1\cap\mc M}=\max\left\{ \|x\|_1,\|x\|_\ii\right\},
\]
\[
(L^1 + L^\ii)(\mc M) = L^1(\mc M) + \mc M \text{ \ \ with \ \ }
\]
\[
\|x\|_{L^1+\mc M}=\inf \left \{ \|y\|_1+ \|z\|_{\infty}: \ x = y + z, \ y \in L^1(\mc M), \ z \in \mc M \right \}=
\int_0^1 \mu_t(x) dt
\]
(see \cite[Proposition 2.5]{ddp}).

Let us notice that if $x \in L^1+\mc M$, then there exists $\lb>0$ such that  $x\, \{|x|>\lb\}\in L^1(\mc M)$.

It is known that a  fully   symmetric  space  $(E(\mc M), \| \cdot \|_{E(\mc M)})$ is an exact interpolation space for the Banach couple $(L^1(\mc M),\mc M)$  \cite{ddp1}. Therefore $T(E(\mc M))\su E(\mc M)$ and $\| T\|_{{E(\mc M)}\to{E(\mc M)}} \leq 1$ for every fully symmetric space $E(0,\ii)$ and any $T\in DS$.

We say that a positive linear operator $T:L^1\to L^1$ is an {\it absolute contraction} and write $T\in AC^+$ if
\[
\| T(x)\|_1\leq \| x\|_1 \ \ \forall \ x\in L^1 \text{ \  and \ } \| T(x)\|_\ii\leq \|x\|_\ii \ \ \forall \ x \in  L^1\cap\mc M.
\]
It is clear that  if $T\in DS^+$, then $T| L^1\in AC^+$. It turns out that any $T\in AC^+$ can be uniquely extended to a positive Dunford-Schwartz operator:

\begin{pro}\label{p611} \cite[Proposition 1.1]{cl1}.
Given $T\in AC^+$, there exists a unique $\wt T \in DS^+$ such that $\wt T|L^1=T$ and $\wt T|\mc M$ is $\sigma(\mc M, L^1)$-continuous.
\end{pro}

Recall that $\{ T_{\mb u}\}_{\mb u\in \mathbb R_+^d}\su DS^+$ is a strongly continuous semigroup on $L^1$ and that $A_t(x)=\frac1{t^d}\int_{[0,t]^d}T_{\mb u}(x)d\mb u \in L^1$ for all $x\in L^1$, $t >0$. Clearly, $\{A_t\}_{t>0}\su AC^+$. By Proposition \ref{p611}, given $t>0$, there exists a unique $\wt A_t \in DS^+$ such that $\wt A_t|L^1 = A_t$ and $\wt A_t|\mc M$ is $\sigma(\mc M, L^1)$-continuous. Let us denote this extension also by $A_t$.

Since a fully symmetric space $E=(E(\mc M), \| \cdot \|_{E(\mc M)})$ is an exact interpolation space for the Banach couple
$(L^1, \mathcal M)$, it now follows that $A_t(E)\su E$ and $\|A_t\|_{E \to E} \leq 1$ for every $t>0$.

Define
\[
\mathcal R_\tau= \{x \in L^1 + \mathcal M: \ \mu_t(x) \to 0 \text{ \ as \ } t\to \ii\}.
\]
In is known that $R_\tau$ is the closure of $ L^1\cap \mc M$ in $L^1 + \mc M$ in the measure topology $t_\tau$ \cite[Proposition 2.7]{ddp}. Since the convergence in the norm $\|\cdot\|_{L^1 + \mc M}$ implies the convergence with respect to $t_\tau$ \cite[Proposition 2.2]{ddp}, it follows that $\mc R_\tau$ is a closed subspace in
$(L^1 + \mc M, \|\cdot\|_{L^1 + \mc M})$. Consequently, $(\mc R_\tau, \|\cdot\|_{L^1+\mc M})$ is a Banach space.
Note that definitions of $\mc R_\tau$ and $\| \cdot \|_{L^1+L^\ii}$ yield that if
$$
x\in \mc R_\tau, \ y\in L^1+\mc M \text{ \ and \ } y \prec\prec x,
$$
then $y\in \mc R_\tau$ and $\| y\|_{L^1+\mc M} \leq \| x\|_{L^1+\mc M}$. Therefore
$(\mc R_\tau, \|\cdot\|_{L^1+\mc M})$ is a  fully symmetric space,  so $A_t(\mc R_\mu)\su \mc R_\mu$
and $\| A_t\|_{\mc R_\mu\to R_\mu}\leq 1$ for every $t>0$.

Furthermore, if $\tau(\mb 1)=\ii$, then a fully symmetric space $E\su L^1+\mc M$ is contained in $\mc R_\tau$ if and only if $\mathbf 1 \notin E$ \cite[Proposition 2.2]{cl3}. In particular, if $E(0,\infty)$ is a a separable fully symmetric function space
and $E(\mathcal M,\tau)$ is the corresponding noncommutative fully symmetric space, then $E(\mathcal M,\tau) \subseteq \mathcal R_\tau $. Note also that if $\tau(\mathbf 1) < \ii$, then $\mc M \su L^1$ and $\mathcal R_\tau = L^1$.

Here is an extension of Theorem \ref{t33} to $\mc R_\tau$:

\begin{teo}\label{t34}
Given $x\in \mathcal R_\tau$, the averages (\ref{eq31}) converge a.u. to $x$ as $t \to 0$.
\end{teo}

\begin{proof}
Without loss of generality assume that $x\ge 0$, and let $\{ e_\lb\}_{\lb\ge 0}$ be
the spectral family of $x$. Given $m \in \mathbb N$, denote $x_m= \int_{1/m}^\ii \lb de_\lb$
and $y_m = \int_0^{1/m} \lb de_\lb$. Then $0 \leq y_m \leq {\frac1m} \cdot \mathbf 1$, \ $x_m \in L^1$,
and $x = x_m + y_m$ for all $m$.

Fix $\ep > 0$. By Theorem \ref{t32}, $A_t(x_m) \to  x_m $ a.u. as $t \to 0$ for each $m$.
Therefore, there exists a sequence $\{e_m\}\su \mc P(\mc M)$ such that
$$
\tau(e_m^{\perp})\leq \frac \ep{2^m} \text{ \ \ and \ \ }
\| (A_t(x_m)- x_m)e_m\|_\ii\to 0 \text{ \ \ as \ \ } t \to 0.
$$
Then it follows that
$$
\|(A_t(x_m)-x_m)e_m\|_{\ii} < \frac1m \text{ \ \ for all \ \ }  0 < t < t(m).
$$
Since $\|y_m\|_{\ii} \leq \frac1m$, we have
\begin{equation*}
\begin{split}
\| (A_t(x)-x)e_m\|_{\ii}
&\leq \| (A_t(x_m)-x_m)e_m\|_{\ii} + \| (A_t(y_m)-y_m))e_m\|_{\ii} \\
&<\frac1m + \|A_t(y_m)e_m\|_{\ii} + \| y_me_m\|_{\ii} \leq {\frac3m}
\end{split}
\end{equation*}
for each $m$ and all $0 < t < t(m)$.

Now, if $e = \bigwedge \limits_m e_m$, then
$$
\tau(e^\perp) \leq \ep \text{ \ \ and \ \ } \| (A_t(x)-x)e\|_\ii< \frac3m \text{ \ \ for all \ \ } 0 < t < t(m).
$$
This means that $A_t(x)\to x$ a.u. as $t \to 0$.
\end{proof}

An application of Theorem \ref{t34} to a fully symmetric space yields the following.

\begin{teo}\label{t35}  Let $\tau(\mb 1)=\ii$, and
let $E\su L^1+\mc M$ be a fully symmetric space such that
$\mathbf 1 \notin E$. Then, given $x\in E$, the averages (\ref{eq31}) converge a.u. to $x$ as $t \to 0$.
\end{teo}

In particular, we have  the following.

\begin{teo}\label{t36}
Let $\tau(\mb 1)=\ii$, and let $E\subset L^1+\mc M$ be a fully symmetric space such that $\mathbf 1 \notin E$.
Then the averages $\frac 1t\int_0^tT_s(x)ds$ converge a.u. to $x$ as $t \to 0$ for every $x\in E$.
\end{teo}

A  symmetric space $E\su L^1+\mc M$ is said to have {\it order continuous norm} if
\[
\| x_n\|_E\downarrow 0 \text{ \ \ whenever\ \ } 0 \leq x_n\in E \text{\ \ and\ \ } x_n\downarrow 0.
\]
If $E$  is a symmetric space  with order continuous norm, then
$\tau\{|x| > \lb\}< \ii$  for all $x\in E$ and $\lb > 0$, so $E\su\mc R_\mu $; in particular, $\mathbf 1\notin E$. It is clear that the spaces $L^p(\mc M)$, $1\leq p<\ii$, have order continuous norms.

\begin{lm}\label{l61}
If $E$ is a symmetric space with order continuous norm, then, given $\ep>0$, there exists $\dt> 0$ such that $\|e\|_E<\ep$ for every $e\in \mc P(\mc M)$ with $\tau(e)<\dt$.
\end{lm}
\begin{proof}
Let $\{e_n\}_{n=1}^\ii \su\mc P(\mc M)$ be such that $0<\tau(e_n)<\ii$ for every $n$ and $e_n\downarrow 0$. Then
$\{e_n\}_{n=1}^\ii\su L^1\cap \mc M \subset E$ and, by order continuity of the norm $\|\cdot\|_E$, we have $\|e_n\|_E \downarrow 0$. Consequently, $\|e_{n_0}\|_E<\ep$ for some $n_0$. Set $\delta = \tau(e_{n_0})$. Then if $e\in\mc P(\mc M)$ and $\tau(e)<\dt$, we obtain
\[
\mu_t(e) =\chi_{[0,\tau(e))} \leq \chi_{[0,\delta)}= \mu_t(e_{n_0}),
\]
hence $\|e\|_E \leq \|e_{n_0}\|_E<\ep$.
\end{proof}

Now, using Theorem \ref{t35}, we give an extension of Theorem \ref{t33b} to fully symmetric spaces with order continuous norms:

\begin{teo}\label{t37}
Let $(E,\|\cdot\|_{E})$ be a fully symmetric space  with order continuous norm. Then $\| A_t(x)-x\|_E\to 0$ as $t\to 0$
for every $x\in E$.
\end{teo}
\begin{proof}
Without loss of generality assume that $x\geq 0$, and let $\{ e_\lb\}_{\lb\ge 0}$ be the spectral family of $x$. Given
$m\in\mathbb N$, denote
\[
a_m= \int\limits_0^{1/m}\lb de_\lb + \int\limits_{m}^\ii  \lb de_\lb, \ \  \ x_m= \int\limits_{1/m}^m\lb de_\lb.
\]
Then $a_m\in E$, $x_m\in E\cap L^1\cap \mc M$, $\|x_m\|_\ii\leq m$ and $x = a_m+x_m$ for all $m$.

Fix $\ep>0$. Since the norm $\|\cdot\|_E$ is order continuous and $a_m \downarrow 0$, it follows that $\|a_m\|_E \downarrow 0$. Consequently, there exists $m _0$ such that $\|a_{m_0}\|_E \leq\ep$, hence $\|A_t(a_{m_0})\|_E\leq\ep$ for all $t >0$.

Lemma \ref{l61} implies that there is $\dt_0>0$ such that the inequality $\tau(e)<\dt_0$ entails $\|e\|_E<\frac{\ep}{m_0}$,
$e\in\mc P(\mc M)$. By Theorem \ref{t35}, $A_t(x_{m_0}) \to  x_{m_0}$ a.u. as $t\to 0$. Therefore, there exists a projection
$e_0\in \mc P(\mc M)$ such that
\[
 \tau(e_0^\perp)<\dt_0 \text{ \ \ and \ \ }
\|(A_t(x_{m_0})- x_{m_0})e_0\|_\ii\to 0 \text{ \ \ as\ \ } t \to 0.
\]
It follows then that
\[
\|(A_t(x_{m_0})- x_{m_0})e_0^\perp\|_E \leq \|(A_t(x_{m_0})- x_{m_0})\|_\ii\, \|e_0^\perp\|_E<
2\ep \text{ \ \ for all \ \ } t >0.
\]
As $x_{m_0}\in L^1$, Theorem \ref{t33b} implies that we also have
\[
\|(A_t(x_{m_0})-x_{m_0})e_0\|_1\leq \|A_t(x_{m_0})-x_{m_0}\|_1\to 0 \text{ \ \ as\ \ }t \to 0.
\]
Consequently,
\[
\|(A_t(x_{m_0})-x_{m_0})e_0\|_{L^1\cap \mc M } \to 0\text{ \ \ as\ \ } t \to 0,
\]
and, since the symmetric space $L^1\cap \mc M$ is continuously embedded in the symmetric space
$(E,\|\cdot\|_E)$, we conclude that
\[
\|(A_t(x_{m_0})- x_{m_0})e_1\|_E \to 0 \text{ \ \ as\ \ } t \to 0.
\]
In particular, there exists $t(\ep)>0$ such that
\[
\|(A_t(x_{m_0})- x_{m_0})e_0\|_E<\ep \text{ \ \ whenever\ \ } 0<t< t(\ep).
\]
Therefore, given $0 < t < t(\ep)$, we have
\begin{equation*}
\begin{split}
\|A_t(x)- x\|_E &\leq \|(A_t(x_{m_0})- x_{m_0})e_0\|_E+\|(A_t(x_{m_0})- x_{m_0})e_0^\perp\|_E\\
&+\|A_t(a_{m_0})\|_E + \|a_{m_0}\|_E < 5\ep,
\end{split}
\end{equation*}
which completes the proof.
\end{proof}

\section{Applications}
In what follows we describe applications of Theorems \ref{t35} and \ref{t37} to noncommutative Orlicz, Lorentz and Marcinkiewicz spaces  and to noncommutative symmetric spaces with order continuous norm. Naturally,
we assume that $\tau(\mb 1)=\ii$.

1. Let $\Phi$ be an Orlicz function, that is, $\Phi:[0,\infty)\to [0,\infty)$ \ is a convex continuous at $0$ function
such that $\Phi(0)=0$ and $\Phi(u)>0$ if $u\ne 0$. Let
\[
L^\Phi=L^\Phi(\mathcal  M, \tau)=\left \{ x \in L^0(\mathcal  M, \tau): \ \tau \left (\Phi\left (\frac {|x|}a \right )\right )
<\infty \text { \ for some \ } a>0 \right \}
\]
be the corresponding {\it noncommutative Orlicz space},  and let
\[
\| x\|_\Phi=\inf \left \{ a>0:  \tau \left (\Phi\left (\frac {|x|}a \right )\right )  \leq 1 \right \}
\]
be the {\it Luxemburg norm} in $L^\Phi$  (see, for example  \cite{cl2}). If  $\tau(\mathbf 1) = \infty$, then $\tau \left (\Phi\left (\frac {\mathbf 1}a \right )\right )= \infty$ for all $a>0$, hence $\mathbf 1  \notin  L^\Phi$. Therefore, by Theorem \ref{t35}, for a given $x\in L^\Phi$, the averages (\ref{eq31}) converge a.u. to $x$ as $t \to 0$.

An Orlicz function $\Phi$ is said that to satisfy {\it $(\Delta_2)$-condition at $0$ (at $\ii$)} if there exist $u_0\in (0,\ii)$ and $k >0$ such that $\Phi(2u)<k\, \Phi(u)$ for all $0<u<u_0$ (respectively, for all $u >u_0$).
An Orlicz function $\Phi$ satisfies $(\Delta_2)$-condition at $0$ and at $\infty$ if and only if $(L^\Phi(0,\ii), \| \cdot\|_\Phi)$ has order continuous norm \cite[Ch.2, \S 2.1, Theorem 2.1.17]{es}. In this case, the corresponding noncommutative Orlicz space also has order continuous norm \cite[Proposition 3.6]{ddp}.

Therefore, Theorem \ref{t37} implies that if an Orlicz function $\Phi$ satisfies $(\Delta_2)$-condition at $0$ and at $\ii$, then $\| A_t(x)-x\|_\Phi\to 0$ as $t\to 0$ for every $x\in L^\Phi$.

2.  Let $\varphi$ be a concave function on $[0,\ii)$ with $\varphi(0) = 0$ and $\varphi(t)>0$ for all $t>0$, and let
\[
\Lb_\varphi=\Lb_\varphi(\mc M, \tau) =\left \{x \in  L^0(\mc M, \tau): \ \|x \|_{\Lb_\varphi} =
\int_0^\ii \mu_t(x)d\varphi(t)<\ii\right \},
\]
be the corresponding {\it noncommutative Lorentz space} (see, for example \cite{cl1}, \cite{cs}).

It is well-known that $(\Lb_\varphi, \|\cdot\|_{\Lb_\varphi})$ is a fully symmetric space; in addition, if $\varphi(\ii) = \ii$, then
$\mathbf 1\notin\Lb_\varphi$ and if $\varphi(\infty) < \ii$, then $\mathbf 1\in\Lambda_\varphi$. Therefore, in the case
$\varphi(\ii) = \ii$, for a given $x\in\Lb_\varphi$, the averages (\ref{eq31}) converge a.u. to $x$ as $t \to 0$.

By \cite[Ch.II, \S 5, Lemma 5.1]{kps}, \cite[Ch.9, \S 9.3, Theorem 9.3.1]{rmgp}), the space $(\Lb_\varphi(0,\ii),\| \cdot \|_{\Lb_\varphi})$, and hence the noncommutative Lorentz space $\Lb_\varphi(\mc M, \tau)$ \cite[Proposition 3.6]{ddp}) has order continuous norm if and only if $\varphi(+0) = 0$ and $\varphi(\ii) = \ii$. Therefore, Theorem \ref{t37} entails that if  
$\varphi$ is a concave function on $[0, \infty)$ with $\varphi(0) = 0$ and $\varphi(t) > 0$ for all $t>0$ such that 
$\varphi(+0) = 0$ and $\varphi(\ii)=\ii$, then $\|A_t(x)-x\|_{\Lb_\varphi}\to 0$ as $t\to 0$  for every 
$x\in \Lb_\varphi(\mc M, \tau) $.

3. Let $\varphi$ be as above, and let
\[
M_\varphi=M_\varphi(\mc M,\tau) = \left \{x \in  L^0(\mc M, \tau): \ \|x \|_{M_\varphi} =
\sup\limits_{0<s< \infty}\frac1{\varphi(s)} \int_0^{s} \mu_t(x) d t<\ii\right\},
\]
be the corresponding {\it noncommutative Marcinkiewicz space}. It is known that $(M_\varphi, \|\cdot\|_{M_\varphi})$
is a fully symmetric space; in addition, $\mathbf 1\notin M_\varphi$ if and only if $\lim\limits_{t \to \ii}\frac{\varphi(t)}t=0$.

Therefore, Theorem \ref{t35} implies that if $\lim\limits_{t\to\ii}\frac{\varphi(t)}t=0$, then
for a given $x\in M_\varphi$, the averages (\ref{eq31}) converge a.u. to $x$ as $t\to 0$.

If $\varphi(+0)>0$ and $\varphi(\ii) < \ii$, then $M_\varphi(\mc M,\tau) = L^1(\mc M,\tau)$ as the sets. In this case, the norms $\|\cdot \|_{M_\varphi}$ and $\|\cdot \|_1$  are equivalent \cite[Ch.6, \S 6.1, Proposition 6.1.1]{rmgp}, and we have,  by  Theorem \ref{t37}, that $\|A_t(x)-x\|_{M_\varphi}\to 0$ as $t\to 0$ for every $x\in M_\varphi(\mc M,\tau)$.



\begin{thebibliography}{99}

\bibitem{br} A. Brunel, The\' orem\` e ergodique ponctuel pour un semigroupe commutatif finiment
engendr\' e de contractions de $L^1$, AIHP B {\bf 9} (1973), 327-343.

\bibitem {ca} E. Castrandas, A local ergodic theorem in  semifinite von Neumann algebras,
{\it Algebras Groups and Geometries}, {\bf 13}  (1996), 71-80.

\bibitem{cls} V. Chilin, S. Litvinov, and A. Skalski, A few remarks in non-commutative ergodic theory, {\it J. Operator Theory}, {\bf 53}(2) (2005), 331-350.

\bibitem{cl1} V. Chilin, S. Litvinov, Ergodic theorems in fully symmetric spaces of $\tau$-measurable operators,
{\it Studia Math.}, {\bf 288}(2) (2015), 177-195.

\bibitem{cl2} V. Chilin, S. Litvinov,  Individual ergodic theorems in noncommutative Orlicz spaces, {\it Positivity}, {\bf  21}(1) (2017), 49-59.

\bibitem{cl3} V. Chilin, S. Litvinov, Individual ergodic theorems for semifinite von Neumann algebras,
arXiv:1607.03452 [math.OA].

\bibitem{cs} V. Chilin, F. Sukochev, Weak convergence in non-commutsative symmetrc spaces, {\it J. Operator Theory},
{\bf 31} (1994), 35-55.

\bibitem{cd} J. P. Conze, N. Dang-Ngoc, Ergodic theorems for noncommutative dynamical systems, {\it Invent. Math.},
{\bf 46} (1978), 1--15.

\bibitem{es} G.A. Edgar and L. Sucheston, {\it Stopping Times and Directed Processes}, Cambridge University Press. 1992.

\bibitem{ddp1} P. G. Dodds, T. K. Dodds, and B. Pagter, Fully symmetric operator spaces, {\it Integr Equat Oper Theory}, {\bf 15} (1992), 942--972.

\bibitem{ddp} P. G. Dodds, T. K. Dodds, and B. Pagter, Noncommutative K\"{o}the duality,
{\it Trans. Amer. Math. Soc.}, {\bf 339}(2) (1993), 717--750.


\bibitem{fk} T. Fack, H. Kosaki, Generalized $s$-numbers of $\tau$-measurable operators, {\it Pacific. J. Math.}, {\bf 123} (1986), 269-300.

\bibitem{li} S. Litvinov, Uniform equicontinuity of sequences of measurable operators and non-commutative
ergodic theorems, {\it Proc. Amer. Math. Soc.}, {\bf 140}(7) (2012), 2401-2409.

\bibitem{jx} M. Junge, Q. Xu, Noncommutative maximal ergodic theorems, {\it J. Amer. Math. Soc.}, {\bf 20}(2) (2007), 385-439.

\bibitem{kps} S. G. Krein, Ju. I. Petunin, and E. M. Semenov, {\it Interpolation of Linear Operators},
Translations of Mathematical Monographs, Amer. Math. Soc., {\bf 54}, 1982.

\bibitem {ks}  N. J. Kalton, F. A. Sukochev, Symmetric norms and spaces of operators, {\it J. Reine Angew. Math.}, {\bf 621} (2008), 81-121.

\bibitem {kr} U. Krengel, {\it Ergodic Theorems}, Walter de Gruyer, Berlin-New York, 1985.

\bibitem{ne} E. Nelson, Notes on non-commutative integration, {\it J. Funct. Anal.}, {\bf 15} (1974), 103-116.

\bibitem{px} G. Pisier, Q. Xu, Noncommutative $L^p$-spaces, in: {\it Handbook of the geometry of Banach spaces},
{\bf 2} (2003), 1459-1517.

\bibitem{rmgp} B. A. Rubshtein, M. A. Muratov,  G. Ya. Grabarnik, Yu. S. Pashkova.  Foundations of Symmetric Spaces of Measurable Functions. Lorentz, Marcinkiewicz and Orlicz Spaces. {\it Developments in Mathematics}. V. 45. Springer International Publishing Switzerland. 2016.

\bibitem{wa} S. Watanabe, Ergodic theorems for dynamical semi-groups on operator algebras, {\it Hokkaido Math. J.},
{\bf 8} (1979), 176-190.

\bibitem{ye0}  F. J. Yeadon, Non-commutative $L^p$-spaces, {\it Math. Proc. Camb. Phil. Soc.},  {\bf 77} (1975), 91-102.

\bibitem{ye} F. J. Yeadon, Ergodic theorems for semifinite von Neumann algebras. I, {\it J. London Math. Soc.}, {\bf 16}(2) (1977), 326-332.


\end{thebibliography}
\end{document}